\documentclass[reqno]{amsart}
\usepackage{hyperref}
 \usepackage{amsmath} 
 \usepackage{amssymb} 

\begin{document}
\title[ GENERALIZED  ZK EQUATION ]
{THE IVP FOR A CERTAIN DISPERSION GENERALIZED \\ ZK EQUATION IN BI-PERIODIC SPACES}

\author[Carolina Albarracin, Guillermo Rodriguez-BLANCO]
{Carolina Albarracin, Guillermo Rodriguez-Blanco}

\address{Carolina Albarracin \newline
Departamento de Matematicas,
 Universidad Nacional de Colombia, 
Carrera 30, calle 45, Bogot\'a, Colombia}
\email{calbarracinh@unal.edu.co}

\address{Guillermo Rodriguez-Blanco \newline
Departamento de Matematicas,
 Universidad Nacional de Colombia, 
Carrera 30, calle 45, Bogot\'a, Colombia}
\email{grodriguezb@unal.edu.co}

\keywords{Well-posedness, Sobolev spaces, Zakharov-Kutnesov equation
 \hfill\break\indent }

\begin{abstract}
 We establish well-posedness conclusions for the Cauchy problem associated to the dispersion generalized Zakharov-Kutnesov equation in bi-periodic  Sobolev spaces  $H^{s}\left(\mathbb{T}^{2}\right)$, $s>(\frac{3}{2}-\frac{1}{2^{\alpha+2}})(\frac{3}{2}-\frac{\beta}{4})$.
\end{abstract}

\maketitle
\numberwithin{equation}{section}
\newtheorem{theorem}{Theorem}[section]
\newtheorem{lemma}[theorem]{Lemma}
\newtheorem{definition}[theorem]{Definition}
\newtheorem{proposition}[theorem]{Proposition}
\newtheorem{remark}[theorem]{Remark}
\allowdisplaybreaks

\section{Introduction}
In the present paper, we deal with the well-posedness of the initial value problem (IVP):
\begin{align}\label{bipe}
	\begin{cases}
		\partial_{t}u-\partial_{x}\left(D_{x}^{1+\alpha}\pm D_{y}^{1+\beta}\right)u+uu_{x}=0 & (x,y)\in\mathbb{T}^{2},\; t\in\mathbb{R}, \\
		u\left(0\right)=\phi & \phi\in H^{s}(\mathbb{T}^{2}),
	\end{cases}
\end{align}
with $\alpha=1,2\;or\; 3$ and $\;0<\beta\leq1$, where  $\mathbb{T}^{2}=\mathbb{R}^{2}/(2\pi\mathbb{Z})^{2}$ and $D^{\beta}_{y}=(
-\partial_{y}^{2})^{\frac{\beta}{2}}$ is a homogeneous fractional derivative in the variable $y$, and it is defined via the Fourier transform by $\widehat{(D^{\beta}_{y}f)}(m,n):
=|n|^{\beta}\hat{f}(m,n)$, analogously is defined $D_{x}^{\alpha}$ for the variable $x$. When,  $\alpha=\beta=1$ and the + sign, this equation
corresponds to the well-known Zakharov-Kuznetsov (ZK) equation, that is a natural bi-dimenmsional extension of the well-known Korteweg de Vries (KdV) equation. The ZK equation describes the propagation of nonlinear ion-acoustic
waves in magnetized plasma and it's well-posedness has been extensively studied in
\cite{linares2009well, linares2010well, faminskii1995cauchy,  biagioni2003well, cunha2014ivp}. 
Recently, models that  generalize the ZK equation have emerged, see \cite{latorre2006evolution}. An important case in the study of this type of equations is the well-posedness in the periodic Sobolev space. In this regard, Linares et al. in \cite{linares2018periodic} obtained  well-posedness for the ZK equation in the periodic Sobolev spaces $H^{s}(\mathbb{T}^{2})$, for $s>\frac{5}{3}$.
Improved by Schippa in \cite{schippa2019cauchy}  to $s>\frac{3}{2}$ via short-time
bilinear Strichartz estimates adapting the bilinear arguments   in the periodic Sobolev space $H^{s}(\mathbb{T}^{2})$. 
Recently, Kinoshita and Schippa  proved local well-posedness in  $H^{s}(\mathbb{T}^{2})$ for $s>1$ in \cite{kinoshita2021loomis},  
the ingredient to improve on previous results
is a nonlinear Loomis-Whitney-type inequality. 

For $\alpha=0$,   $\beta=1$ and the + sign, the equation (\ref{bipe}) coincides with  the Benjamin-Ono-Zakharov-Kuznetsov (BOZK) equation that
is a model for thin nano-conductors on a dielectric
substrate in   \cite{latorre2006evolution}. For result concerning  well-posedness (see \cite{esfahani2009instability, esfahani2011ill,cunha2016ivp}). When $\alpha=1,3,5 \cdots$ the equation (\ref{bipe})  in one dimension  is,
\begin{align}\label{disaltas}
	\partial_{t}u-\partial_{x}^{\alpha}u+uu_{x}=0,
\end{align}

that is known as the KdV equation with higher dispersion(see \cite{gorsky2009well} and the references contained therein). Results regarding  local 
well-posedness for (\ref{disaltas}) in $H^{s}(\mathbb{T})$, $s\geq-\frac{1}{2}$ were obtained in \cite{gorsky2009well}, using bilinear
estimates in Bourgain spaces. Unlike our work, $s$ does
not depend on the order of the dispersion term. 
Another similar equation studied in $\mathbb{R}^{2}$ is the  two-dimensional dispersive generalized g-BOZK equation (see \cite{ribaud2016local}), in this case the general dispersion is in $x$ with $0\leq\alpha\leq1$ and $\beta =1$, where local well-posedness was established in $H^{s}(\mathbb{R}^{2})$ for $s>\frac{2}{1+\alpha}-\frac{3}{4}$.

We observe that, the IVP (\ref{bipe}) satisfies at least following conserved quantities:
\begin{align}\label{masa}
	M(u)=\int_{\mathbb{T}^{2}}u^{2}dxdy
\end{align}
and
\begin{align}\label{energia}
	E(u)=\frac{1}{2}\int_{\mathbb{T}^{2}}\left(\left(D^ {\frac{1+\alpha}{2}}_{x}u\right)^{2}\pm\left(D_{y}^{\frac{1+\beta}{2}}u\right)^{2}-\frac{u^{3}}{3}\right)dxdy,
\end{align}

so that (\ref{masa}) and (\ref{energia}) are usefull to extend the local solution to global one, so to attain the global well-posedness in anisotropic Sobolev spaces $H^{\frac{1+\alpha}{2},\frac{1+\beta}{2}}(\mathbb{T}^{2})$. Unfortunately, we are not dealing with the well-posedness in this spaces here.
Our goal in this work is to improve the local well-posedness in periodic Sobolev space $H^{s}(\mathbb{T}^{2})$, $s>2$ for (\ref{bipe}), which is obtained from a parabolic regularization argument,  following Iorio's ideas in \cite{iorio2001fourier} Chapther 6. One way to improve this result is to use again energy estimates in smooth solutions, what leads us to obtain control of the norms  $\left\Vert\nabla u\right
\Vert_{L^{1}_{T}L^{\infty}_{xy}}$. If the Sobolev embedding is used, we achieve this estimate, but we cannot improve regularity.
Therefore, we use  the short-time Strichartz linear approach introduced by Koch and Tzvetkov in  \cite{koch2003local}, to get local well-posedness of  Benjamin-Ono (BO) equation in $\mathbb{R}$, which has been proved to be useful for these types of equations (see  \cite{hickman2019higher,linares2018cauchy,kenig2004local,kenig2003local}).
 But to perform this task in  two dimensions with periodic context,  we adapt the method used to prove local well-posedness for the Cauchy problem associated to the third-order KP-I and fifth-order KP-I equations on $\mathbb{R}\times\mathbb{T}$ and $\mathbb{T}^{2}$ proposed by Ionescu and Kenig in  \cite{ionescu2009local} (for other applications, see \cite{linares2010well, bustamante2019periodic} and the references therein).
 First, a localized Strichartz-type
estimate for the linear part of the equation
is obtained, where the main difficulty lies in
obtaining bounds for exponential sums in
the periodic case (see \cite{graham1991van}), such sums have been treated in different contexts in number theory. Then fixing $\alpha$, we can apply a lemma due to H. Weyl (Lemma \ref{weyl} below)  
that combined with some Strichartz estimates give a  control to $\left\Vert u\right\Vert_{L^{1}_{T}L^{\infty}_{xy}}$, $\left\Vert u_{x}\right\Vert_{L^{1}_{T}L^{\infty}_{xy}}$ and $\left\Vert u_{y}\right\Vert_{L^{1}_{T}L^{\infty}_{xy}}$. Finally by  standard compactness methods, we get the result.
 Although
 the method used by  Kinoshita and  Shippa  \cite{kinoshita2021loomis} produces better results for the ZK equation in the bi-periodic setting, it is not clear how to extended this technique to equations involving  fractional operators in both spatial  variables $x$ and $y$. 
 
We  will now give the precise statement of our result in spaces of lower regularity,

\begin{theorem}\label{teoimpo}
	\label{teo26}
	Let $\alpha=1,2\;or\;3$, $0<\beta\leq1$,  $\phi\in H^{s}(\mathbb{T}^{2})$ and $s>(\frac{3}{2}-\frac{1}{2^{\alpha+2}})(\frac{3}{2}-\frac{\beta}{4})$ such that $\int_{\mathbb{T}}\phi(x,y)  dx=0$ a. e. $y\in \mathbb{T}$,  there exist  $T=T\left(\left\Vert \phi\right\Vert _{H^{s}}\right)$ and a unique solution of IVP (\ref{bipe}), such that  $u\in C\left(\left[0,T\right];H^{s}(\mathbb{T}^{2})\right)$ and  $u,\:\partial_{x}u,\:\partial_{y}u\in L_{T}^{1}L_{xy}^{\infty}$. Moreover, the map data-solution $\phi\in H^{s}(\mathbb{T}^{2})\mapsto u\in C\left(\left[0,T\right];H^{s}(\mathbb{T}^{2})\right)$ is continuous.
\end{theorem}
  
Where it is observed that when the 
dispersion in the $x$ variable
increases the regularity required 
increases slightly.

The paper is organized as follows. In the following section we prove Localized Strichartz Estimate. Section 3, deals with Preliminary and Key estimates. Lastly, the main result is proved. Before setting our results, some notation is necesary:

\textbf{\textit{Notation}}. 
\begin{itemize}
	\item $a\lesssim b$ (resp. $a\gtrsim b$) means that
	there exists a positive constant $c$, such that, $a \leq cb$ (resp. $a\geq cb$). 
	\item $a\sim  b$, when $a\lesssim b$ and $a\gtrsim b$. 
	\item $C^{p}(X)$ for the function space $C^{p}$ class in X.
	\item \begin{align*}
		\left\Vert u\right\Vert_{L^{p}X}:=\left(\int_{\mathbb{R}}\left(\left\Vert u(t)\right\Vert_{X}\right)^{p}dt\right)^{\frac{1}{p}}\\
		\left\Vert u\right\Vert_{L^{\infty}X}:=ess\;sup_{t\in\mathbb{R}}\left\Vert u(t)\right\Vert_{X},
	\end{align*}
	where $X$ is Banach space,  $u:\mathbb{R} \to X$ is a measurable  function and    $p\in[1,\infty]$.
	\item \begin{align*}
		\left\Vert u\right\Vert_{L^{p}_{I}X}:= \left\Vert \chi_{I}(|t|) u\right\Vert_{L^{p}X},
	\end{align*}
	where $I\subseteq\mathbb{R}$ is interval, $\chi_{I}$ is the characteristic function of $I$ and $u:I \to X$ is a measurable  function. In particular, for $T>0$
	\begin{align*}
		\left\Vert u\right\Vert_{L^{p}_{T}X}:=\left\Vert u\right\Vert_{L^{p}_{[-T,T]}X}.
	\end{align*}
	\item 
	\begin{align*}
		\widehat{f}\left(m,n\right)= C\int_{\mathbb{T}^{2}}f\left(x,y\right)e^{-ixm}e^{-iyn}dxdy  , 
	\end{align*}
	
	where  $f\in L^{1}\left(\mathbb{T}^{2}\right)$, $\left(m,n\right)\in\mathbb{Z}^{2}$ and $C$ as a universal constant. It is the Fourier transform of $f$.
	\item  $\mathcal{S}=C^{\infty}(\mathbb{T}^{2})$.
	\item $\mathcal{S}'$ is the biperiodic distribution space.
	\item $\widehat{f}\left(m,n\right)=\langle f;\mathfrak{e}_{-m-n}\rangle$  is the Fourier transform of $f\in\mathcal{S}'$, where $\mathfrak{e}_{mn}(x,y)=e^{i(mx+ny)}$, $(m,n)\in\mathbb{Z}^{2}$, $(x,y)\in\mathbb{T}^{2}$,  $C$ as a universal constant and $\langle;\rangle$ is the duality bracket of $\mathcal{S}'$,$\mathcal{S}$.
	\item
	\begin{equation*}
		H^s(\mathbb{T}^{2}) = \{f \in \mathcal{S}' \, : \, \sum_{(m,n)\in\mathbb{Z}^{2}}(1 + m^2+n^2)^s \, |\hat{f}(m,n)|^2 \,   <\infty\},
	\end{equation*}
	where $H^s(\mathbb{T}^{2})$ denote the standard Sobolev spaces in $L^2(\mathbb{T}^{2})$  for  $s\in\mathbb{R}$,
	and
	\begin{align*}
		H^{\infty}\left(\mathbb{T}^{2}\right)=\underset{s\geq0}{\cap}H^{s}\left(\mathbb{T}^{2}
		\right).
	\end{align*}
\begin{equation*}
	\left\Vert f\right\Vert _{H^{s}}\sim\left\Vert J_{y}^{s}f\right\Vert _{L^{2}\left(\mathbb{T}\times\mathbb{T}\right)}+\left\Vert J_{x}^{s}f\right\Vert _{L^{2}\left(\mathbb{T}\times\mathbb{T}\right)} 
\end{equation*}
	\item
	For integers $k=0,1,\cdots$ we define the operators $Q_{x}^k$ and $Q_{y}^k$ on $H^\infty(\mathbb{T}^{2})$ by
	\begin{equation}
		\begin{cases}
			\widehat{Q_{x}^{0}g}\left(m,n\right)=\chi_{\left[0,1\right)}\left(\left|m\right|\right)\hat{g}\left(m,n\right)\\
			\widehat{Q_{x}^{k}g}\left(m,n\right)=\chi_{\left[2^{k-1},2^{k}\right)}\left(\left|m\right|\right)\hat{g}\left(m,n\right) & si\:k\geq1
		\end{cases}
	\end{equation}
	and
	\begin{equation}
		\begin{cases}
			\widehat{Q_{y}^{0}g}\left(m,n\right)=\chi_{\left[0,1\right)}\left(\left|n\right|\right)\hat{g}\left(m,n\right)\\
			\widehat{Q_{y}^{k}g}\left(m,n\right)=\chi_{\left[2^{k-1},2^{k}\right)}\left(\left|n\right|\right)\hat{g}\left(m,n\right) & si\:k\geq1,
		\end{cases}
	\end{equation}

	where $\chi_{I}$ is the characteristic function over $I$.
	\item We observe  the following  equivalence for the Sobolev norms on $\mathbb{T}^{2}$,
	\begin{align}
		\left\Vert g
		\right\Vert^{2}_{L^{2}(\mathbb{T}\times\mathbb{T})}\sim\sum_{k,j\geq 0}\left\Vert Q_{x}^{k}Q_{y}^{j}g\right\Vert^{2}_{L^{2}(\mathbb{T}\times\mathbb{T})}
	\end{align}
and
		\begin{align}
		\left\Vert J^{s}_{x} g
		\right\Vert^{2}_{L^{2}(\mathbb{T}\times\mathbb{T})} +\left\Vert J^{s}_{y} g
		\right\Vert^{2}_{L^{2}(\mathbb{T}\times\mathbb{T})} &\sim\sum_{k\geq0,j\geq 1}(2^{j})^{2s}\left\Vert Q_{x}^{k}Q_{y}^{j}g\right\Vert^{2}_{L^{2}(\mathbb{T}\times\mathbb{T})}\\&\nonumber+	\sum_{k\geq1,j\geq 0}(2^{k})^{2s}\left\Vert Q_{x}^{k}Q_{y}^{j}g\right\Vert^{2}_{L^{2}(\mathbb{T}\times\mathbb{T})}\\&\nonumber +\sum_{k,j\geq 0}\left\Vert Q_{x}^{k}Q_{y}^{j}g\right\Vert^{2}_{L^{2}(\mathbb{T}\times\mathbb{T})}
	\end{align}
\end{itemize}

\section{Localized Strichartz
	Estimate}
In this section, we prove Strichatrz estimate localized in frequency and time. First we recall the following lemmas.
\begin{lemma}
	[Poisson Summation Formula]\label{sumpoisson}
	Let $f$, $\hat{f}$ are in $L^{1}(\mathbb{R}^{n})$ and  satisfy the condition  
	\begin{align*}
		\left|f\left(x\right)\right|+\left|\hat{f}\left(x\right)\right|\leq C\left(1+\left|x\right|\right)^{-n-\delta}   \end{align*}
	for some constant $C$, $\delta>0$. Then   
	$f$ and  $\hat{f}$ are continuous an for all $x\in R^{n}$ we have
	\begin{align*}
		\sum_{m\in\mathbb{Z}^{n}}\widehat{f}\left(2\pi m\right)=\sum_{k\in\mathbb{Z}^{n}}f\left( k\right) .  
	\end{align*}
	
\end{lemma}
\begin{proof}
	See \cite{grafakos2008classical}  Theorem 3.1.17
\end{proof}

\begin{lemma}[Van der Corput]\label{vander}
	Let $p\geq2$, $I
	=[a,b]$ $\varphi\in C^{p}\left(I\right)$ be a real value  function  such that $\left|\varphi^{\left(p\right)}\left(x\right)\right|\geq\lambda>0$,
	$\psi\in L^{\infty}\left(I\right)$ and $\psi'\in L^{1}\left(I\right)$. Then,
	\begin{align*}
	  \left|\int_{I}e^{i\varphi\left(x\right)}\psi\left(x\right)dx\right|\leq C_{p}\lambda^{\frac{1}{p}}\left(\left\Vert \psi\right\Vert _{L^{\infty}}+\left\Vert \psi'\right\Vert _{L^{1}}\right)  \end{align*}
\end{lemma}
\begin{proof}
	See \cite{stein1993harmonic} Chapter 8.
\end{proof}
\begin{lemma}\label{weyl}
	If $h\left(x\right)=\omega_{d}x^{d}+...+\omega_{1}x+\omega_{0}$ is
	a polynomial with real coefficients and  $\left|\omega_{d}-\frac{a}{q}\right|\leq\frac{1}{q^{2}}$
	for some  $a\in\mathbb{Z}$ and  $q\in\mathbb{Z}^{+}$ with $\left(a,q\right)=1$ ($a$ and $b$ are relatively prime)
	then for any  $\delta>0$,
	\begin{align*}
		\left|{\sum}_{m=1}^{N}e^{2\pi ih\left(m\right)}\right|\leq C_{\delta,d}N^{1+\delta}\left[q^{-1}+N^{-1}+qN^{-d}\right]^{\frac{1}{2^{d-1}}}   
	\end{align*}
	where the constant $C_{\delta,d}$ only denpends on $\delta$ and $d$.
	
\end{lemma}
\begin{proof}
	See \cite{nathanson2013additive} Theorem 4.3.
\end{proof}
\begin{lemma}\label{weyl2}
	For any integer $\Lambda\geq1$ and any $r\in\mathbb{R}$, there are integers $q\in\left\{1,2,...,\Lambda\right\}$ and $a\in \mathbb{Z}$, $(a,q)=1$, such that
	\begin{align*}
		\left|r-\frac{a}{q}\right|\leq\frac{1}{\Lambda q}
	\end{align*}
\end{lemma}
\begin{proof}
	It is consequence of Dirichlet's principle, see \cite{nathanson2013additive} Theorem 4.1.
\end{proof}
Let
$\left\{W_{0}^{\alpha}\left(t\right)\right\}_{t}$ be the group associated with the linear part of the equation (\ref{bipe}). If $\int_{0}^{t}u_{0}(x,y)dx=0$ a. e. $y\in\mathbb{T}$, then, 

\begin{align*}
	W^{\alpha}_{0}\left(t\right)\phi&=\sum_{m\in\mathbb{Z}^{*}}\sum_{n\in\mathbb{Z}}\widehat{\phi}\left(\textbf{m}\right)e^{i\left[\left(\textbf{m}\textbf{x}\right)+m\left(\left|m\right|^{1+\alpha}\pm \left|n\right|^{1+\beta}\right)t\right]}+\sum_{n\in\mathbb{Z}}\widehat{\phi}\left(0,n\right)e^{iny}\\&=\sum_{m\in\mathbb{Z}^{*}}\sum_{n\in\mathbb{Z}}\widehat{\phi}\left(\textbf{m}\right)e^{i\left[\left(\textbf{m}\textbf{x}\right)+m\left(\left|m\right|^{1+\alpha}\pm \left|n\right|^{1+\beta}\right)t\right]},
\end{align*}

where  $\boldsymbol{m}=\left(m,n\right)$, $\boldsymbol{x}=\left(x,y\right)$ and $\boldsymbol{m}\cdot\boldsymbol{x}=mx+ny$. 
\begin{theorem}
	\label{striper}
	Let   $\alpha=1,2\;or\; 3 $,  $0<\beta\leq1$ and $\phi\in L^{2}\left(\mathbb{T}^{2}\right)$  
	, then for any $\epsilon>0$, 
	\[
	\left\Vert W_{0}^{\alpha}\left(.\right)Q_{y}^{k}Q_{x}^{j}\phi\right\Vert _{L^{2}_{2^{-(k+j)}}L^{\infty}\left(\mathbb{T}^{2}\right)}\lesssim_{\epsilon}2^{\left(-\frac{1}{2^{\alpha+2}}+\epsilon\right)j+\left(-\frac{\beta}{4}+\epsilon\right)k} \left\Vert Q_{y}^{k}Q_{x}^{j}\phi\right\Vert _{L^{2}\left(\mathbb{T}^{2}\right)}.
	\]
\end{theorem}
\begin{proof}
	Let $\psi_{1}:\mathbb{R}\rightarrow\left[0,1\right]$ denote a smooth even function supported in  $\left\{ r\mid\left|r\right|\in\left[\frac{1}{4},4\right]\right\} $  and    $\psi_{1}\equiv1$
	in  $\left\{ r\mid\left|r\right|\in\left[\frac{1}{2},2\right]\right\} $. Let 
	$a\left(\boldsymbol{m}\right)=\left(Q_{y}^{k}Q_{x}^{j}\phi\right)^{\wedge}\left(\boldsymbol{m}\right)$
	, $\psi_{1}\left(\frac{m}{2^{j}}\right)\cdot\psi_{1}\left(\frac{n}{2^{k}}\right)=1$ in $\left[-2^{j+1},-2^{j-1}\right]\times\left[-2^{k+1},-2^{k-1}\right]\cup\left[2^{j-1},2^{j+1}\right]\times\left[2^{k-1}, 2^{k+1}\right]$ and    $supp\;Q_x^jQ_y^k\subset supp\;\psi^{j}_{1}\psi^{k}_{1}$, $j,k\geq1$. Then,  
	\begin{align*}
	W_{0}^{\alpha}\left(t\right)Q_{y}^{k}Q_{x}^{j}\phi=\sum_{(m,n)\in\mathbb{Z^{*}}\times\mathbb{Z}}a\left(\boldsymbol{m}\right)\psi_{1}\left(\frac{m}{2^{j}}\right)\psi_{1}\left(\frac{n}{2^{k}}\right)e^{i\left[\left(\boldsymbol{m}\cdot\boldsymbol{x}\right)+m\left(\left|m\right|^{1+\alpha}\pm \left|n\right|^{1+\beta}\right)t\right]}
	\end{align*}
	It suffice to prove that,
	\begin{align*}
		\left\Vert \chi_{\left[0,2^{-(k+j)}\right]}\left(\left|t\right|\right)\sum_{(m,n)\in\mathbb{Z^{*}}\times\mathbb{Z}}a\left(\boldsymbol{m}\right)\psi_{1}\left(\frac{m}{2^{j}}\right)\psi_{1}\left(\frac{n}{2^{k}}\right)e^{i\left[\left(\boldsymbol{m}\cdot\left(x\left(t\right),y\left(t\right)\right)\right)+m\left(\left|m\right|^{1+\alpha}\pm\left|n\right|^{1+\beta}\right)t\right]}\right\Vert _{L_{t}^{2}}
		\\ \leq C_{\epsilon}2^{\left(-\frac{1}{2^{\alpha+2}}+\epsilon\right)j+\left(-\frac{\beta}{4}+\epsilon\right)k}\left\Vert a\right\Vert _{l^{2}\left(\mathbb{Z}^{2}\right)}.
	\end{align*}
	By duality, for any $g\in L_{t}^{2}$,
	$g:\mathbb{R}\rightarrow\mathbb{R}$
	\begin{align*}
		\left\Vert \int_{\mathbb{R}}g\left(t\right)\chi_{\left[0,2^{-(k+j)}\right]}\left(\left|t\right|\right)\psi_{1}\left(\frac{m}{2^{j}}\right)\psi_{1}\left(\frac{n}{2^{k}}\right)e^{i\left[\left(\boldsymbol{m}\cdot\left(x\left(t\right),y\left(t\right)\right)\right)+m\left(\left|m\right|^{1+\alpha}\pm\left|n\right|^{1+\beta}\right)t\right]}dt\right\Vert _{l^{2}\left(\mathbb{Z}^{2}\right)}\\
		\leq C_{\epsilon}2^{\left(-\frac{1}{2^{\alpha+2}}+\epsilon\right)j+\left(-\frac{\beta}{4}+\epsilon\right)k}\left\Vert g\right\Vert _{L_{t}^{2}},
	\end{align*}
	for any measurable functions $x,y:\left[-2^{-(k+j)},2^{-(k+j)}\right]\rightarrow\mathbb{T}$. By expanding the $L^2$-norm on the left-hand, we get,
	\begin{align*}
		\left|\int_{\mathbb{R}}\int_{\mathbb{R}}g\left(t\right)g\left(t'\right)K_{k+j}\left(t,t',x,y\right)dtdt'\right|\leq C_{\epsilon}2^{\left(-\frac{1}{2^{\alpha+1}}+2\epsilon\right)j+\left(-\frac{\beta}{2}+2\epsilon\right)k}   \end{align*}
	with
	\begin{align*}
		K_{k+j}\left(t,t',x,y\right)&:=\chi_{\left[0,2^{-(k+j)}\right]}\left(\left|t\right|\right)\chi_{\left[0,2^{-(k+j)}\right]}\left(\left|t'\right|\right)\\&\sum_{(m.n)\in\mathbb{Z^{*}}\times\mathbb{Z}}\psi_{1}^{2}\left(\frac{m}{2^{j}}\right)\psi_{1}^{2}\left(\frac{n}{2^{k}}\right)e^{i\left[\left(\boldsymbol{m}\cdot\left(\boldsymbol{x}\left(t\right)-\boldsymbol{x}\left(t'\right)\right)\right)+m\left(\left|m\right|^{1+\alpha}\pm\left|n\right|^{1+\beta}\right)\left(t-t'\right)\right]}.   
	\end{align*}
	For integers $l\geq k+j$, we define;
	\begin{align*}
		K_{k+j}^{l}\left(t,t',x,y\right):=\chi_{\left(2^{-l},2\cdot2^{-l}\right]}\left(\left|t-t'\right|\right)K_{k+j}\left(t,t',x,y\right)   
	\end{align*}
	
	with $t\neq t'$ and $\;t\;,t'\;\in\left[-2^{-(k+j)},2^{-(k+j)}\right]$. Thus, it is enough to prove that,
	\begin{align}\label{desprovar}
	\left|\sum_{m\in\mathbb{Z^{*}}}\sum_{n\in\mathbb{Z}}\psi_{1}^{2}\left(\frac{m}{2^{j}}\right)e^{i\left[mx+tm\left|m\right|^{1+\alpha}\right]}\psi_{1}^{2}\left(\frac{n}{2^{k}}\right)e^{i\left[ny\pm tm\left|n\right|^{1+\beta}\right]}\right|\lesssim2^{l}2^{\left(-\frac{1}{2^{\alpha+1}}+2\epsilon\right)j+\left(-\frac{\beta}{2}+2\epsilon\right)k},   
	\end{align}
	for any 
	$x,y\in\left[0,2\pi\right)$ and   $\left|t\right|\in\left[2^{-l},2^{-l}2\right]$. The cases $k,j=0$ is inmediate, if  $j,k>0$  using the Poisson summation formula  (Lemma \ref{sumpoisson}) in the suma in $n$, we get  
	\begin{align*}
	\sum_{m\in\mathbb{Z^{*}}}\psi_{1}^{2}\left(\frac{m}{2^{j}}\right)e^{i\left[mx+tm\left|m\right|^{1+\alpha}\right]}\left(\sum_{\nu\in\mathbb{Z}}\int_{\mathbb{R}}\psi_{1}^{2}\left(\frac{\eta}{2^{k}}\right)e^{i\left[\left(y-2\pi\nu\right)\eta\pm tm\left|\eta\right|^{1+\beta}\right]}d\eta\right).   
	\end{align*}
	
	We will use integration by parts  to solve the integral term in the above expression.  We define $A:=supp\;\psi_1\left(\frac{\eta}{2^{k}}\right)=\left\{\eta : 2^{k-1}\leq\left|\eta\right|\leq2^{k+1}\right\}$,
	\begin{align}\label{integralesacotar}
		&\nonumber \int_{\mathbb{R}}\psi_{1}^{2}\left(\frac{\eta}{2^{k}}\right)e^{i\left(\left(y-2\pi\nu\right)\eta \pm tm\left|\eta\right|^{1+\beta}\right)}d\eta\\&\nonumber
		=\int_{A}\frac{\psi_{1}^{2}\left(\frac{\eta}{2^{k}}\right)}{i\left[\left(y-2\pi\nu\right)\pm\left(1+\beta\right)sgn\left(\eta\right)\left|\eta\right|^{\beta}mt\right]}\frac{d}{d\eta}\left(e^{i\left(\left(y-2\pi\nu\right)\eta \pm tm\left|\eta\right|^{1+\beta}\right)}\right)d\eta\\
		& =\frac{-1}{i}\int_{A}\left[\frac{2\cdot2^{-k}\psi_{1}\psi_{1}^{'}}{\left[\left(y-2\pi\nu\right)\pm\left(1+\beta\right)sgn\left(\eta\right)\left|\eta\right|^{\beta}mt\right]}\right.\\&\nonumber
		\left.\pm\frac{\psi_{1}^{2}\left(1+\beta\right)\beta\left|\eta\right|^{\beta-1}mt}{\left[\left(y-2\pi\nu\right)\pm\left(1+\beta\right)sgn\left(\eta\right)\left|\eta\right|^{\beta}mt\right]^{2}}\right] 
		\left(e^{i\left(\left(y-2\pi\nu\right)\eta \pm tm\left|\eta\right|^{1+\beta}\right)}\right)d\eta
	\end{align}
	For the second term of the previous integral, we have that, $\left(1+\beta\right)\beta\left|\eta\right|^{\beta-1}mt\leq4\left(2\cdot2^k\right)^{\beta-1}2^{j}2^{-(j+k)}\lesssim 2^{(\beta-2)k}$, 
	$\left(1+\beta\right)\left|\eta\right|^{\beta}mt\leq2\cdot 2^{k\beta}2^{j}2^{-(k+j)} \lesssim 2$ and    
	$y\in\left[0,2\pi\right)$, $\left|\nu\right|>100$, then $\left|\left(y-2\pi\nu\right)\pm\left(1+\beta\right)sgn\left(\eta\right)\left|\eta\right|^{\beta}mt\right|^{2}\sim \left|\nu\right|^{2}$. We get,
\begin{align*}
	\left|\int_{A}\frac{\psi_{1}^{2}(1+\beta)\beta|\eta|^{\beta-1}m t}{\left[\left(y-2\pi\nu\right)\pm (1+\beta)sgn(\eta)|\eta|^{\beta} m t\right]^{2}}
	e^{i\left(\left(y-2\pi\nu\right)\eta \pm tm|\eta|^{1+\beta}\right)}d\eta\right|<C2^{k}\frac{2^{(\beta-2) k}}{\left|\nu\right|^{2}}=C\frac{1}{\left|\nu\right|^{2}}
\end{align*}

	On the other hand, for  the first term of the right-hand side of  (\ref{integralesacotar}), we again use integration by parts,
	\begin{align*}
		&\int_{A}\left[\frac{2\cdot2^{-k}\psi_{1}\psi_{1}^{'}}{\left[\left(y-2\pi\nu\right) \pm \left(1+\beta\right)sgn\left(\eta\right)\left|\eta\right|^{\beta}mt\right]}\right]\left(e^{i\left(\left(y-2\pi\nu\right)\eta \pm tm\left|\eta\right|^{1+\beta}\right)}\right)d\eta\\
		&=\frac{1}{i}\int_{A}\left[\frac{2\cdot2^{-k}\psi_{1}\psi_{1}^{'}}{\left[\left(y-2\pi\nu\right) \pm \left(1+\beta\right)sgn\left(\eta\right)\left|\eta\right|^{\beta}mt\right]^{2}}\right]\frac{d}{d\eta}\left(e^{i\left(\left(y-2\pi\nu\right)\eta \pm tm\left|\eta\right|^{1+\beta}\right)}\right)d\eta\\
		&=\frac{-1}{i}\int_{A}\left[\frac{2\cdot2^{-2k}\left(\psi_{1}^{'}\right)^{2}+2\cdot2^{-2k}\psi_{1}\psi_{1}^{''}}{\left[\left(y-2\pi\nu\right)\pm\left(1+\beta\right)sgn\left(\eta\right)\left|\eta\right|^{\beta}mt\right]^{2}}\right.\\&\left.\pm \frac{4\psi_{1}\psi_{1}^{'}2^{-k}\left(1+\beta\right)\beta\left|\eta\right|^{\beta-1}mt}{\left[\left(y-2\pi\nu\right)\pm\left(1+\beta\right)sgn\left(\eta\right)\left|\eta\right|^{\beta}mt\right]^{3}}\right] 
		\left(e^{i\left(\left(y-2\pi\nu\right)\eta \pm tm\left|\eta\right|^{1+\beta}\right)}\right)d\eta.
	\end{align*}
	
	Following similar consideration to the second term of  (\ref{integralesacotar}). We can conclude that, if $\left|\nu\right|>100$, 
	$\left|\int_{\mathbb{R}}\psi_{1}^{2}\left(\frac{\eta}{2^{k}}\right)e^{i\left(\left(y-2\pi\nu\right)\eta
	\pm tm\left|\eta\right|^{1+\beta}\right)}d\eta\right|\leq\frac{C}{\left|\nu\right|^{2}}+\frac{C}{\left|\nu\right|^{3}}\leq\frac{C}{\left|\nu\right|^{2}}$. Then,
	\begin{align*}
		\sum_{\nu}\int_{\mathbb{R}}\psi_{1}^{2}\left(\frac{\eta}{2^{k}}\right)e^{i\left[\left(y-2\pi\nu\right)\eta \pm tm\left|\eta\right|^{1+\beta}\right]}d\eta\\=\sum_{\left|\nu\right|\leq100}\int_{\mathbb{R}}\psi_{1}^{2}\left(\frac{\eta}{2^{k}}\right)e^{i\left[\left(y-2\pi\nu\right)\eta \pm tm\left|\eta\right|^{1+\beta}\right]}d\eta+O\left(1\right).  
	\end{align*}
	
	So it is enough to estimate,
	\begin{align*}
		\left|\sum_{m=1}^{\infty}\psi_{1}^{2}\left(\frac{m}{2^{j}}\right)e^{i\left[mx+tm^{2+\alpha}\right]}\int_{\mathbb{R}}\psi_{1}^{2}\left(\frac{\eta}{2^{k}}\right)e^{i\left[y'\eta
		\pm tm\left|\eta\right|^{1+\beta}\right]}d\eta\right|
		&\lesssim2^{l}2^{\left(-\frac{1}{2^{\alpha+1}}+2\epsilon\right)j+\left(-\frac{\beta}{2}+2\epsilon\right)k},
	\end{align*}
	with $\left|t\right|\in\left[2^{-l},2^{-l+1}\right]$, $m\sim2^{k}$ and   $0\leq\beta\leq1$, being the estimate of the other sum similar.
	The estimation of the 
	following oscillatory integral, is a consequence of Lemma \ref{vander} with  $\varphi(\eta)=y'\eta \pm tm\left|\eta\right|^{1+\beta}$ and $|\eta| \sim2^{k}$,
	\begin{align}\label{oscila}
		\left|\int_{\mathbb{R}}\psi_{1}^{2}\left(\frac{\eta}{2^{k}}\right)e^{i\left[y'\eta \pm tm\left|\eta\right|^{1+\beta}\right]}d\eta\right|
		\leq\frac{2^{\frac{1-\beta}{2}k}}{\left|mt\right|^{\frac{1}{2}}}.
	\end{align}
Now, we use the summation by parts formula,
\begin{align*}
	\sum_{m=1}^{\infty}a_{m}b_{m}=\sum_{N=1}^{\infty}\left(\sum_{m=1}^{N}a_{m}\right)\left(b_{N}-b_{N+1}\right)   
\end{align*}
for any compactly supported sequences $a_{m}$ and $b_{m}$. To get,
	\begin{align}\label{perstri}
		&\left|\sum_{m=1}^{\infty}\psi_{1}^{2}\left(\frac{m}{2^{j}}\right)e^{i\left(mx+tm^{2+\alpha}\right)}
		\int_{\mathbb{R}}\psi_{1}^{2}\left(\frac{\eta}{2^{k}}\right)e^{i\left(y'\eta \pm tm\left|\eta\right|^{1+\beta}\right)}d\eta\right|\\&\nonumber\hspace{0.5cm}=\left|\sum_{N=1}^{\infty}\left(\sum_{m=1}^{N}a_{m}\right)\left(b_{N}-b_{N+1}\right)\right|.
	\end{align}
	where  $a_{m}=e^{i\left(mx+tm^{2+\alpha}\right)}$,  $b_{m}=\psi_{1}^{2}\left(\frac{m}{2^{j}}\right)\int_{\mathbb{R}}\psi_{1}^{2}\left(\frac{\eta}{2^{k}}\right)e^{i\left(y'\eta \pm tm\left|\eta\right|^{1+\beta}\right)}d\eta$ and $N\in\left[2^{j-1},2^{j+1}\right]$. 
	\begin{align*}
	\left|b_{N}-b_{N+1}\right|&\leq\left|\psi_{1}^{2}\left(\frac{N}{2^{j}}\right)-\psi_{1}^{2}\left(\frac{N+1}{2^{j}}\right)\right|\left|\int_{\mathbb{R}}\psi_{1}^{2}\left(\frac{\eta}{2^{k}}\right)e^{i\left[y'\eta \pm tN\left|\eta\right|^{1+\beta}\right]}d\eta\right|
	\\&+\left|\psi_{1}^{2}\left(\frac{N+1}{2^{j}}\right)\right|	
	\left|\int_{\mathbb{R}}\psi_{1}^{2}\left(\frac{\eta}{2^{k}}\right)e^{i\left[y'\eta \pm tN\left|\eta\right|^{1+\beta}\right]}\left(e^{\pm it\left|\eta\right|^{1+\beta}}-1\right)d\eta\right|
	\\&\leq2^{-j}\frac{2^{\frac{1-\beta}{2}k}}{\left|Nt\right|^{\frac{1}{2}}}
	+|t|\int_{\mathbb{R}}\left|\eta^{\frac{1+\beta}{2}}\psi_{1}\left(\frac{\eta}{2^{k}}\right)\right|^{2}d\eta
		\\&\leq2^{-j}\frac{2^{\frac{1-\beta}{2}k}}{\left|Nt\right|^{\frac{1}{2}}}
	+|t|2^{k}\left\Vert D^{\frac{1+\beta}{2}}\hat{\psi_{1}}\right\Vert_{0}^{2}
		\\&\leq2^{-j}\frac{2^{\frac{1-\beta}{2}k}}{\left|Nt\right|^{\frac{1}{2}}}
	+|t|2^{k}\left\Vert\hat{\psi_{1}}\right\Vert_{\frac{1+\beta}{2}}^{2}
	\\&\leq2^{l}2^{-2j}2^{-\frac{\beta}{2}k}+2^{l}2^{-2j}2^{-k}
	\end{align*}
    where $\left|t\right|\in\left[2^{-l},2^{-l+1}\right]$, $N\sim2^{j}$, $l\geq k+j$ and $\left\Vert\hat{\psi_{1}}\right\Vert_{\frac{1+\beta}{2}}<\infty$.  
   Then we must estimate,
	\begin{align}\label{remplzo}
		\left|\sum_{N=2^{j-1}}^{2^{j+1}}\left(\sum_{m=1}^{N}e^{i\left(mx+tm^{2+\alpha}\right)}\right)(b_{N}-b_{N+1})\right|&\nonumber\leq 2^{l}2^{-2j}2^{-\frac{\beta}{2}k}\sum_{N   =2^{j-1}}^{2^{j+1}}\left|\sum_{m=1}^{N}a_{m}\right|\\&\lesssim 2^{l}2^{-j}2^{-\frac{\beta}{2}k}\left|\sum_{m=1}^{N}a_{m}\right|		
	\end{align}
	
	Following the proof of the Teorema 9.3.1 de \cite{ionescu2009local}, we use Lemma of H. Weyl  in $\left|\sum_{m=1}^{N}a_{m}\right|$. If $\alpha=1$,     $h\left(m\right)=\frac{t}{2\pi}m^{3}+\frac{x}{2\pi}m$, we fix $\Lambda=2^{2j+5}$ and apply Lemma \ref{weyl2} to $r=\frac{t}{2\pi}$ then, $\left|\frac{t}{2\pi}-\frac{a}{q}\right|\leq\frac{1}{2^{2j+5}q} $.
	Since $j$ is large,  $\:N\sim2^{j}$, $l\in[2j,2j+2]$
	(the restriction $l\leq2j+2$ guarantees that $\frac{a}{q}\neq0$) and $\;2^{j}\leq q\leq2^{2j+5}$, then by Lemma \ref{weyl}, we get 
	\begin{align*}
	\left|\sum_{m=1}^{N}e^{i\left(mx+tm^{3}\right)}\right|\leq
	C_{\epsilon} (2^{j})^{1+2\epsilon}\left(\frac{1}{2^{j}}+\frac{1}{2^{j}}+\frac{2^{2j+5}}{2^{3j}}\right)^{\frac{1}{4}}\lesssim_{\epsilon} 2^{\left(\frac{3}{4}+2\epsilon\right) k}.	
	\end{align*}  
	 When $\alpha=2$,   $h\left(m\right)=\frac{t}{2\pi}m^{4}+\frac{x}{2\pi}m$, we fix $\Lambda=2^{3j}$ and apply Lemma \ref{weyl2} to $r=\frac{t}{2\pi}$ then, $\left|\frac{t}{2\pi}-\frac{a}{q}\right|\leq\frac{1}{2^{3j}q} $.
	Since $j$ is large, $\:N\sim2^{j}$, $l\in[2j,\frac{5}{2}j]$ and $\;2^{
		j}\leq q\leq2^{3j}$, then by Lemma \ref{weyl}. We get,   
	$\left|\sum_{m=1}^{N}e^{i\left(mx+tm^{4}\right)}\right|\leq C2^{\left(\frac{7}{8}+2\epsilon\right) j}$. And if $\alpha=3$,   $h\left(m\right)=\frac{t}{2\pi}m^{5}+\frac{x}{2\pi}m$, we fix $\Lambda=2^{4j}$ and apply Lemma \ref{weyl2} to $r=\frac{t}{2\pi}$ then, $\left|\frac{t}{2\pi}-\frac{a}{q}\right|\leq\frac{1}{2^{4j}q} $.
	Since $j$ is large, $\:N\sim2^{j}$, $l\in[2j,3j]$ and $\;2^{
		j}\leq q\leq2^{4j}$, then by Lemma \ref{weyl}. We get,   
	$\left|\sum_{m=1}^{N}e^{i\left(mx+tm^{5}\right)}\right|\leq C2^{\left(\frac{15}{16}+2\epsilon\right) j}$. 
	In the case $\alpha=1$. By replacing in (\ref{remplzo}),  we get,
	\begin{align*}
		\left|\sum_{m=1}^{\infty}\psi_{1}^{2}\left(\frac{m}{2^{k}}\right)e^{i\left[mx+tm^{1+\alpha}\right]}\int_{\mathbb{R}}\psi_{1}^{2}\left(\frac{\eta}{2^{k}}\right)e^{i\left[y'\eta+tm\left|\eta\right|^{1+\beta}\right]}d\eta\right|\lesssim2^{l}2^{\left(-\frac{1}{4}+2\epsilon\right)j+\left(\frac{-\beta}{2}+2\epsilon\right)k},
	\end{align*}
Analogously in the other cases.
\end{proof}
\begin{remark}
	\begin{itemize}
		\item
	
	As in the case $\alpha=2$, we have,
	
	\begin{align*}
		\sum_{m\in \mathbb{Z^{*}}}\psi_{1}^{2}\left(\frac{m}{2^{j}}\right)e^{i\left[mx+ts
			gn(m)m^{4}\right]}\int_{\mathbb{R}}\psi_{1}^{2}\left(\frac{\eta}{2^{k}}\right)e^{i\left[y'\eta
		\pm	tm\left|\eta\right|^{1+\beta}\right]}d\eta
		&\\
		=\sum_{m=1}^{\infty}\psi_{1}^{2}\left(\frac{m}{2^{j}}\right)e^{i\left[mx+tm^{4}\right]}\int_{\mathbb{R}}\psi_{1}^{2}\left(\frac{\eta}{2^{k}}\right)e^{i\left[y'\eta
		\pm
		tm\left|\eta\right|^{1+\beta}\right]}d\eta&\\+
		\sum_{m=-1}^{-\infty}\psi_{1}^{2}\left(\frac{m}{2^{j}}\right)e^{i\left[mx-tm^{4}\right]}\int_{\mathbb{R}}\psi_{1}^{2}\left(\frac{\eta}{2^{k}}\right)e^{i\left[y'\eta
		\pm
		tm\left|\eta\right|^{1+\beta}\right]}d\eta.
	\end{align*}
	So it is enough to estimate,
	\begin{align*}
		\left|\sum_{m=1}^{\infty}\psi_{1}^{2}\left(\frac{m}{2^{j}}\right)e^{i\left[mx+tm^{4}\right]}\int_{\mathbb{R}}\psi_{1}^{2}\left(\frac{\eta}{2^{k}}\right)e^{i\left[y'\eta \pm tm\left|\eta\right|^{1+\beta}\right]}d\eta\right|
		&\lesssim2^{l}2^{\left(-\frac{1}{8}+2\epsilon\right)j+\left(-\frac{\beta}{2}+2\epsilon\right)k}.
	\end{align*}
\item In the case   $\widehat{Q^{0}_{y}Q^{j}_{x}\phi}$, $j>0$,   we used  the summation by parts formula, with $a_{m}=e^{i\left(mx+tm^{2+\alpha}\right)}$,    $b_{m}=\psi_{1}^{2}\left(\frac{m}{2^{j}}\right)$ and  Lemma of H. Weyl, 
\begin{align*}
	\left\Vert W_{0}\left(.\right)Q_{y}^{0}Q_{x}^{j}\phi\right\Vert _{L^{2}_{2^{-j}}L^{\infty}\left(\mathbb{T}^{2}\right)}\lesssim_{\epsilon} 2^{\left(-\frac{1}{2^{\alpha+2}}+\epsilon\right)j}\left\Vert Q_{y}^{0}Q_{x}^{j}\phi\right\Vert _{L^{2}\left(\mathbb{T}^{2}\right)}.
\end{align*}
The case, $\widehat{Q^{k}_{y}Q^{0}_{x}\phi}$, $k\geq0$, is not contemplated, because, if  $\hat{\phi}(0,n)=0$, for all $n\in\mathbb{Z}$, then, $\hat{u}(0,n,t)=0$, for all $n\in\mathbb{Z}$ .
	\end{itemize}
\end{remark}
\section{Preliminary and Key estimates}
As a consequence of the Strichartz inequality (Theorem \ref{striper}), we obtain:
\begin{lemma}\label{unot}
	Let $\alpha=1,2\;or\;3$,  $0<\beta\leq1$, $u\in C\left(\left[0,T\right];H^{\infty}\left(\mathbb{T}^{2}\right)\right)\cap C^{1}\left(\left[0,T\right];H^{\infty}\left(\mathbb{T}^{2}\right)\right)$ and  
	$f\in C\left(\left[0,T\right];H^{\infty}\left(\mathbb{T}^{2}\right)\right)$,
	$T\in\left[0,1\right]$ such that
	\[
	\partial_{t}u-\partial_{x}\left(D_{x}^{1+\alpha}\pm D_{y}^{1+\beta}\right)u=\partial_{x}f.
	\]
	Then, 	
	\begin{equation}
		\left\Vert u\right\Vert _{L_{T}^{1}L_{xy}^{\infty}\left(\mathbb{T}^{2}\right)}\lesssim_{s_{1}s_{2}}T^{\frac{1}{2}}\left(\left\Vert J_{x}^{s_{1}}J_{y}^{s_{2}}u\right\Vert _{L_{T}^{\infty}L_{xy}^{2}\left(\mathbb{T}^{2}\right)}+\left\Vert J_{x}^{s_{1}}f\right\Vert _{L_{T}^{1}L_{xy}^{2}\left(\mathbb{T}^{2}\right)}\right),  
	\end{equation}
	for any $s_1>\frac{1}{2}-\frac{1}{2^{\alpha+2}}$ $s_2>\frac{1}{2}-\frac{\beta}{4}$ 
\end{lemma}
\begin{proof}
	We partition the interval $\left[0,T\right]$ into  $2^{j+k}$ equal intervals
	of length $T2^{-\left(j+k\right)}$, denote by $\left[a_{k,m},a_{k,\left(m+1\right)}\right)$,  
	$m=0,1,2\cdots2^{j+k}$. Then,
	\begin{equation}\label{246}
		\left\Vert Q_{y}^{k}Q_{x}^{j}u\right\Vert _{L_{T}^{1}L_{xy}^{\infty}}\leq\sum_{m=1}^{2^{k+j}}\left\Vert \chi_{\left[a_{k,m},a_{k,\left(m+1\right)}\right)}\left(t\right)Q_{y}^{k}Q_{x}^{j}u\right\Vert _{L_{T}^{1}L_{xy}^{\infty}}.
	\end{equation}
	By Cauchy-Schwarz inequality in (\ref{246}), 
	\begin{equation}\label{247}
		\left\Vert \chi_{\left[a_{k,m},a_{k,\left(m+1\right)}\right)}\left(t\right)Q_{y}^{k}Q_{x}^{j}u\right\Vert _{L_{T}^{1}L_{xy}^{\infty}}\lesssim\left(T2^{-(k+j)}\right)^{\frac{1}{2}}\left\Vert \chi_{\left[a_{k,m},a_{k,\left(m+1\right)}\right)}\left(t\right)Q_{y}^{k}Q_{x}^{j}u\right\Vert _{L_{T}^{2}L_{xy}^{\infty}}.  
	\end{equation}
	By  Duhamel's formula, for $t\in\left[a_{k,m},a_{k,\left(m+1\right)}\right]$ ,
	\begin{equation}\label{DUHAMEL1}
		u\left(t\right)=W_{0}^{\alpha}\left(t-a_{k,m}\right)\left(u\left(a_{k,m}\right)\right)+\int_{a_{k,m}}^{t}W_{0}^{\alpha}\left(t-s\right)\left(\partial_{x}f\left(s\right)\right)ds  
	\end{equation}
	It follows from (\ref{DUHAMEL1})  and   Theorem \ref{striper} that,
	\begin{align*}
		&\left\Vert \chi_{\left[a_{k,m},a_{k,\left(m+1\right)}\right)}\left(t\right)Q_{y}^{k}Q_{x}^{j}u\right\Vert _{L_{T}^{2}L_{xy}^{\infty}} \\& \leq\left\Vert \chi_{\left[a_{k,m},a_{k,\left(m+1\right)}\right)}\left(t\right)W_{0}^{\alpha}\left(t-a_{k,m}\right)Q_{y}^{k}Q_{x}^{j}u\left(a_{k,m}\right)\right\Vert _{L_{T}^{2}L_{xy}^{\infty}}\\
		& \hspace{2.6cm}+\left\Vert \int_{a_{k,m}}^{t}\chi_{\left[a_{k,m},a_{k,\left(m+1\right)}\right)}\left(t\right)W_{0}^{\alpha}\left(t-s\right)Q_{y}^{k}Q_{x}^{j}\partial_{x}\left(f\right)ds\right\Vert _{L_{T}^{2}L_{xy}^{\infty}}\\
		& \lesssim_{\epsilon}2^{\left(-\frac{1}{2^{\alpha+2}}+\epsilon\right)j+\left(-\frac{\beta}{4}+\epsilon\right)k}\left\Vert Q_{y}^{k}Q_{x}^{j}u\left(a_{k,m}\right)\right\Vert _{L_{xy}^{2}}
		\\&\hspace{2.6cm}+\int_{a_{k,m}}^{a_{k,m+1}}2^{\left(-\frac{1}{2^{\alpha+2}}+\epsilon\right)j+\left(-\frac{\beta}{4}+\epsilon\right)k}2^{j}\left\Vert Q_{y}^{k}Q_{x}^{j}f\right\Vert _{L_{xy}^{2}}\\
		& \lesssim_{\epsilon}2^{\left(-\frac{1}{2^{\alpha+2}}+\epsilon\right)j+\left(-\frac{\beta}{4}+\epsilon\right)k}\left\Vert Q_{y}^{k}Q_{x}^{j}u\left(a_{k,m}\right)\right\Vert _{L_{xy}^{2}}\\&\hspace{3.5cm}+2^{\left(-\frac{1}{2^{\alpha+2}}+\epsilon\right)j+\left(-\frac{\beta}{4}+\epsilon\right)k}2^{j}\left\Vert \chi_{\left[a_{k,m},a_{k,\left(m+1\right)}\right)}Q_{y}^{k}Q_{x}^{j
		}f\right\Vert _{L_{T}^{1}L_{xy}^{2}}
	\end{align*}
	Then, the left hand side of (\ref{247}) is bounded by,
	\begin{align*}
		&2^{-\frac{k+j}{2}}T^{\frac{1}{2}}
		\sum_{m=1}^{2^{k+j}}\left(2^{\left(-\frac{1}{2^{\alpha+2}}+\epsilon\right)j+\left(-\frac{\beta}{4}+\epsilon\right)k}\left\Vert Q_{y}^{k}Q_{x}^{j}u\left(a_{k,m}\right)\right\Vert _{L_{xy}^{2}}\right.\\&\hspace{4cm}\left.+2^{\left(-\frac{1}{2^{\alpha+2}}+\epsilon\right)j+\left(-\frac{\beta}{4}+\epsilon\right)k}2^{j}\left\Vert \chi_{\left[a_{km},a_{k\left(m+1\right)}\right)}Q_{y}^{k}Q_{x}^{j}f\right\Vert _{L_{T}^{1}L_{xy}^{2}}\right)\\&\leq2^{-\epsilon\frac{(j+k)}{2}}
		T^{\frac{1}{2}}\left(2^{-(k+j)}\sum_{m=1}^{2^{k+j}}\left(\left\Vert 2^{\left(\frac{1}{2}-\frac{\beta}{4}+\epsilon'\right)k}Q_{y}^{k}2^{\left(\frac{1}{2}-\frac{1}{2^{\alpha+2}}+\epsilon'\right)j}Q_{x}^{j}u\left(a_{k,m}\right)\right\Vert _{L_{xy}^{2}}\right)\right.\\&\hspace{4cm}\left.+2^{(\frac{3\epsilon}{2}-\frac{2+\beta}{4})k}\left\Vert Q_{y}^{k}2^{\left(\frac{1}{2}-\frac{1}{2^{\alpha+2}}+\epsilon'\right)j}Q_{x}^{j}f\right\Vert _{L_{T}^{1}L_{xy}^{2}}\right)
	\end{align*}
	where $2^{(\frac{3\epsilon}{2}-\frac{2+\beta}{4})k}<1$ for $k\geq1$  and  $\epsilon<\frac{2+\beta}{6}$, the result is followed.
\end{proof}
To obtain the energy estimate, we recall the periodic version of the Kato -Ponce commutador,
\begin{proposition}
	\label{CONMUTADORBIPE}
	Let $s\geq1$ and $f,g\in H^{\infty}\left(\mathbb{T}^{2}\right)$.
	Then, 
	\begin{equation*}
		\left\Vert J^{s}\left(fg\right)-fJ^{s}g\right\Vert _{0}\leq C_{s} \left\Vert J^{s}f\right\Vert _{0}\left\Vert g\right\Vert _{\infty}+\left(\left\Vert f\right\Vert _{\infty}+\left\Vert \nabla f\right\Vert _{\infty}\right)\left\Vert J^{s-1}g\right\Vert _{0}.   
	\end{equation*}
\end{proposition}
\begin{proof}
	See Lemma 9.A.1 in \cite{ionescu2009local}  
\end{proof}
\begin{lemma}[Energy estimate]\label{supper} Let $\alpha=1,2\;or\;3$, $0<\beta\leq1$ and $u$ solution of IVP (\ref{bipe}) with $\phi\in H^{\infty}\left(\mathbb{T}^{2}\right)$,
	then, for any $s\geq1$, we have, for any  $T\in\left[0,1\right]$, that
	\begin{equation*}
		\underset{0<t<T}{sup}\left\Vert u\right\Vert _{H^{s}\left(\mathbb{T}^{2}\right)}\lesssim e^{\left(\left\Vert u\right\Vert _{L_{T}^{1}L_{xy}^{\infty}\left(\mathbb{T}^{2}\right)}+\left\Vert \partial_{x}u\right\Vert _{L_{T}^{1}L_{xy}^{\infty}\left(\mathbb{T}^{2}\right)}+\left\Vert \partial_{y}u\right\Vert _{L_{T}^{1}L_{xy}^{\infty}\left(\mathbb{T}^{2}\right)}\right)}\left\Vert \phi\right\Vert _{H^{s}\left(\mathbb{T}^{2}\right)}.
	\end{equation*}
\end{lemma}
\begin{proof}We apply $J^{s}$ to equation in (\ref{bipe}) and multiply by $J^{s}u$,
	\begin{equation*}
		\int_{\mathbb{T}^{2}} J^{s}u_{t}J^{s}udxdy-\int_{\mathbb{T}^{2}} J^{s}\partial_{x}D_{x}^{1+\alpha}uJ^{s}udxdy\pm\int_{\mathbb{T}^{2}} J^{s}\partial_{x}D_{y}^{1+\beta}uJ^{s}udxdy
	\end{equation*}
	\begin{align*}
		+\int_{\mathbb{T}^{2}} J^{s}uu_{x}J^{s}udxdy=0.
	\end{align*}
	
	By integrating by parts and applying Proposition \ref{CONMUTADORBIPE}, we get:
	\begin{align*}
		\frac{1}{2}\frac{d}{dt}\left\Vert J^{s}u\right\Vert _{0}^{2}\lesssim\left\Vert \partial_{x}u\right\Vert _{\infty}\left\Vert J^{s}u\right\Vert _{0}+\left(\left\Vert u\right\Vert _{\infty}+\left\Vert \nabla u\right\Vert _{\infty}\right)\left\Vert J^{s-1}\partial_{x}u\right\Vert _{0}.
	\end{align*}
	then, using Gronwall's inequality, we get the result.
\end{proof}
Now, we obtain an estimate of $u$, $u_{x}$ and $u_{y}$  in $L^{1}_{T}L^{\infty}_{xy}$.
\begin{theorem}
	 [Product Lemma] Let $s\geq 0$ and $f, g \in H^{\infty}(\mathbb{T}^
	2)$. Then
	\begin{equation*}
		\left\Vert J^{s}\left(fg\right)\right\Vert _{L^{p}}\leq C \left(\left\Vert J^{s}f\right\Vert _{L^{\infty}}\left\Vert g\right\Vert _{L^{p}}+\left\Vert f\right\Vert _{L^{\infty}}\left\Vert J^{s}g\right\Vert _{L^{p}}\right),  
	\end{equation*}

\end{theorem}
\begin{proof}
	See Lemma 4.2 in \cite{bustamante2019periodic}
\end{proof}
\begin{proposition}\label{constante}
	Let $\alpha=1,2\;or\;3$, $0<\beta\leq1$, $u$  be a  solution del IVP (\ref{bipe}) with  $\phi\in H^{\infty}\left(\mathbb{T}^{2}\right)$, then  for any  $s>(\frac{3}{2}-\frac{1}{2^{\alpha+2}})(\frac{3}{2}-\frac{\beta}{4})$,  there exists  $T=T(\left\Vert \phi\right\Vert _{s},s)$ and  a constant $C_{T}(\left\Vert \phi\right\Vert _{s},s)$ such that,
	\begin{align*}
		g\left(T\right):=\int_{0}^{T}\left(\left\Vert u\right\Vert _{L_{xy}^{\infty}\left(\mathbb{T}^{2}\right)}+\left\Vert u_{x}\right\Vert _{L_{xy}^{\infty}\left(\mathbb{T}^{2}\right)}+\left\Vert u_{y}\right\Vert _{L_{xy}^{\infty}\left(\mathbb{T}^{2}\right)}\right)dt'\leq C_{T}
	\end{align*}
\end{proposition}

\begin{proof}
	We apply Lemma \ref{unot}  with $s_1>\frac{1}{2}-\frac{1}{2^{\alpha+2}}$ and  $s_2>\frac{1}{2}-\frac{\beta}{4}$  in  $u$,
	$\partial_xu$,  $\partial_{y}u$ and respectively by $f=\frac{1}{2}u^{2},\frac{1}{2}\partial_xu^{2},\frac{1}{2}\partial_yu^{2}$. 
	We get,
	\begin{align*}
		&\left\Vert u\right\Vert _{L_{T}^{1}L^{\infty}}+\left\Vert u_{x}\right\Vert _{L_{T}^{1}L^{\infty}}+\left\Vert u_{y}\right\Vert _{L_{T}^{1}L^{\infty}}\\&
		\lesssim_{s}T^{\frac{1}{2}}\left(\left\Vert J_{x}^{s_{1}}J_{y}^{s_{2}}u\right\Vert _{L_{T}^{\infty}L^{2}}+\left\Vert J_{x}^{s_{1}}J_{y}^{s_{2}}\partial_{x}u\right\Vert _{L_{T}^{\infty}L^{2}}+\left\Vert J_{x}^{s_{1}}J_{y}^{s_{2}}\partial_{y}u\right\Vert _{L_{T}^{\infty}L^{2}}\right.\\&+ \left.\left\Vert J_{x}^{s_{1}}u^{2}\right\Vert _{L_{T}^{1}L^{2}}+\left\Vert J_{x}^{s_{1}} \partial_{x}\left(u^{2}\right)\right\Vert _{L_{T}^{1}L^{2}}+\left\Vert J_{x}^{s_{1}} \partial_{y}\left(u^{2}\right)\right\Vert _{L_{T}^{1}L^{2}}\right)  
	\end{align*}
The first three terms above can be bounded by Young's inequality with
	 $p=\frac{s_{2}+1}{s_{2}}$ and $q=s_{2}+1$,  $p=\frac{s_{1}+1}{s_{1}}$  and $q=s_{1}+1$, $p=q=2$ respectively,
	\begin{align*}
		\left\Vert J_{x}^{s_{1}}J_{y}^{s_{2}}\partial_{x}u\right\Vert _{L^{2}} & \lesssim \sum_{(m,n)\in\mathbb{Z}^{2}}\left(1+m^{2}\right)^{s_{1}+1}\left(1+n^{2}\right)^{s_{2}}|\widehat{u}|^{2}\\
		& \lesssim\left\Vert \left(1+m^{2}\right)^{\frac{(s_{1}+1)(s_{2}+1)}{2}}\widehat{u}\right\Vert _{L^{2}}+\left\Vert \left(1+n^{2}\right)^{\frac{s_{2}+1}{2}}\widehat{u}\right\Vert _{L^{2}}\\
		& \lesssim\left\Vert J_{x}^{(s_{2}+1) (s_{1}+1)}u\right\Vert _{L^{2}}+\left\Vert J_{y}^{s_{2}+1}u\right\Vert _{L^{2}} 
		\\&\lesssim\left\Vert u\right\Vert_{s},
	\end{align*}   
	
	\begin{align*}
		\left\Vert J_{x}^{s_{1}}J_{y}^{s_{2}}\partial_{y}u\right\Vert _{L^{2}} & \lesssim \sum_{(m,n)\in\mathbb{Z}^{2}}\left(1+m^{2}\right)^{s_{1}}\left(1+n^{2}\right)^{s_{2}+1}|\widehat{u}|^{2}\\
		& \lesssim\left\Vert \left(1+\xi^{2}\right)^{\frac{s_{1}+1}{2}}\widehat{u}\right\Vert _{L^{2}}+\left\Vert \left(1+n^{2}\right)^{\frac{\left(s_{1}+1\right)\left(s_{2}+1\right)}{2}}\widehat{u}\right\Vert _{L^{2}}\\
		& \lesssim\left\Vert J_{x}^{s_{1}+1}u\right\Vert _{L^{2}}+\left\Vert J_{y}^{(s_{2}+1)(s_{1}+1)}u\right\Vert _{L^{2}}
		\\& \lesssim\left\Vert u\right\Vert_{s} 
	\end{align*}
	and
	\begin{align*}
		\left\Vert J_{x}^{s_{1}}J_{y}^{s_{2}}u\right\Vert _{L^{2}}  =\sum_{(m,n)\in\mathbb{Z}^{2}} \left(1+n^{2}\right)^{s_{1}}\left(1+n^{2}\right)^{s_{2}}|\widehat{u}|^{2}
		& \lesssim\left\Vert \left(1+\xi^{2}\right)^{s_{1}}\widehat{u}\right\Vert _{L^{2}}+\left\Vert \left(1+n^{2}\right)^{s_{2}}\widehat{u}\right\Vert _{L^{2}}\\
		& \lesssim\left\Vert J_{x}^{2s_{1}}u\right\Vert _{L^{2}}+\left\Vert J_{y}^{2s_{2}}u\right\Vert _{L^{2}}
		\\& \lesssim\left\Vert u\right\Vert_{s}
	\end{align*}
Applying Leibniz's product rule periodic version and Young's inequality,
	\begin{align*}
		\int_{0}^{T}\left\Vert J_{x}^{s_{1}}u^{2}\right\Vert _{L^{2}}dt'&\lesssim\int_{0}^{T}\left\Vert u\right\Vert _{L^{\infty}}\left\Vert J_{x}^{s_{1}}u\right\Vert _{L^{2}}dt'\\&\leq\left\Vert u\right\Vert _{L_{T}^{1}L^{\infty}}\left\Vert J_{x}^{s_{1}}u\right\Vert _{L_{T}^{\infty}L^{2}} \\&\leq \left\Vert u\right\Vert _{L_{T}^{1}L^{\infty}}\left\Vert u\right\Vert _{L_{T}^{\infty}H^{s}},
	\end{align*}
	\begin{align*}
		\int_{0}^{T}\left\Vert J_{x}^{s_{1}}\left(u\partial_{x}u\right)\right\Vert _{L^{2}} dt'& \lesssim\int_{0}^{T}\left\Vert \partial_{x}u\right\Vert _{L^{2}}\left\Vert J_{x}^{s_{1}}u\right\Vert _{L^{\infty}}+\left\Vert J_{x}^{s_{1}}\partial_{x}u\right\Vert _{L^{2}}\left\Vert u\right\Vert _{L^{\infty}}dt'\\
		& \leq\left\Vert \partial_{x}u\right\Vert _{L_{T}^{\infty}L^{2}}\left\Vert J_{x}^{s_{1}}u\right\Vert _{L_{T}^{1}L^{\infty}}+\left\Vert J_{x}^{s_{1}}\partial_{x}u\right\Vert _{L_{T}^{\infty}L^{2}}\left\Vert u\right\Vert _{L_{T}^{1}L^{\infty}}\\
		& \leq\left\Vert u\right\Vert _{L_{T}^{\infty}H^{s}}\left\Vert J_{x}^{s_{1}}u\right\Vert _{L_{T}^{1}L^{\infty}}+\left\Vert u\right\Vert _{L_{T}^{\infty}H^{s}}\left\Vert u\right\Vert _{L_{T}^{1}L^{\infty}}
	\end{align*}
	and
	\begin{align*}
		\int_{0}^{T}\left\Vert J_{x}^{s_{1}}\left(u\partial_{y}u\right)\right\Vert _{L^{2}} dt'& \lesssim\int_{0}^{T}\left\Vert \partial_{y}u\right\Vert _{L^{2}}\left\Vert J_{x}^{s_{1}}u\right\Vert _{L^{\infty}}+\left\Vert J_{x}^{s_{1}}\partial_{y}u\right\Vert _{L^{2}}\left\Vert u\right\Vert _{L^{\infty}}dt'\\
		& \leq\left\Vert \partial_{y}u\right\Vert _{L_{T}^{\infty}L^{2}}\left\Vert J_{x}^{s_{1}}u\right\Vert _{L_{T}^{1}L^{\infty}}+\left\Vert J_{x}^{s_{1}}\partial_{y}u\right\Vert _{L_{T}^{\infty}L^{2}}\left\Vert u\right\Vert _{L_{T}^{1}L^{\infty}}\\
		& \leq\left\Vert u\right\Vert _{L_{T}^{\infty}H^{s}}\left\Vert J_{x}^{s_{1}}u\right\Vert _{L_{T}^{1}L^{\infty}}+\left\Vert u\right\Vert _{L_{T}^{\infty}H^{s}}\left\Vert u\right\Vert _{L_{T}^{1}L^{\infty}}.
	\end{align*}

	Then adding the previous inequalities and applying the respective energy estimate (Lemma \ref{supper}) we obtain the inequality,
	\begin{align*}
		g\left(T\right)\lesssim\left\Vert \phi\right\Vert _{s}e^{g\left(T\right)}\left(1+g\left(T\right)\right).
	\end{align*}
	To complete the proof, by an argument of continuity if $T\leq T_{0}C_{T}(\left\Vert \phi\right\Vert _{s},s)$ is small enough,
	$g(T)\leq C_{T}(\left\Vert \phi\right\Vert _{s},s)$
\end{proof}
In this point, we can use standard compactness arguments as in Kenig \cite{kenig2015well} for proving Theorem.

\section{Proof of Theorem 1.1}
Let $s>(\frac{3}{2}-\frac{1}{2^{\alpha+2}})(\frac{3}{2}-\frac{\beta}{4})$,  $\phi\;\in H^{s}\left(\mathbb{T}^{2}\right)$. We consider by density, $\phi_{\gamma}\;\in H^{\infty}\;\cap\; H^{s}$ such that $\underset{\gamma\rightarrow\infty}{lim}\left\Vert \phi_{\gamma}-\phi\right\Vert _{H^{s}\left(\mathbb{T}^{2}\right)}=0$ and $\left\Vert \phi_{\gamma}\right\Vert_{H^{s}\left(\mathbb{T}^{2}\right)}\leq C\left\Vert \phi\right\Vert _{H^{s}\left(\mathbb{T}^{2}\right)}$. Let  $\left\{u_{\gamma}\right\}$ solutions associated to the initial data $\left\{\phi_{\gamma}\right\}$ such that $u_{\gamma}\in C\left(\left[0,T'\right];H^{\infty}\left(\mathbb{T}^{2}\right)\right)$, $T'>0$  guaranteed by the local well-posedness of \ref{bipe} in $H^{s}(\mathbb{T}^{2})$ for $s>2$.  We can extend $u_{\gamma}$ on a time interval $[0,T]$,  $T=T\left(\left\Vert \phi\right\Vert _{H^{s}},s\right)$ by Proposition \ref{constante} and also we show that  there is a constant $C_{T}$  such that,
\begin{align}\label{desinfinito2}
	\int_{0}^{T}\left(\left\Vert u_{\gamma}\right\Vert _{L_{xy}^{\infty}}+\left\Vert \partial_{x} u_{\gamma}\right\Vert _{L_{xy}^{\infty}}+\left\Vert \partial_{y}u_{\gamma}\right\Vert _{L_{xy}^{\infty}}\right)dt\leq C_{T}
\end{align}
We deduce  from energy estimate (Lemma \ref{supper}) and previous inequality  (\ref{desinfinito2}) that,
\begin{align}\label{ct22}
	\underset{0<t<T}{sup}\left\Vert u_{\gamma}\right\Vert _{H^{s}\left(\mathbb{T}^{2}\right)}\leq C_{T}
\end{align}
by using Gronwall's inequality and inequality (\ref{desinfinito2}) we get,
\begin{align}\label{ldos}
	\underset{\gamma,\mu\rightarrow\infty}{lim}\underset{0<t<T}{sup}\left\Vert u_{\gamma}-u_{\mu}\right\Vert _{0}=0  
\end{align}
Now we consider  the inequality (\ref{ct22}) and (\ref{ldos}), we can find  $u\;\in C\left(\left[0,T\right];H^{s_{1}}\left(\mathbb{T}^{2}\right)\right)\;\cap\; L^{\infty}\left(\left[0,T\right];H^{s}\left(\mathbb{T}^{2}\right)\right)$ with $s_{1}<s$, such that $u_{\gamma}\rightarrow u$ in $C\left(\left[0,T\right];H^{s_{1}}\left(\mathbb{T}^{2}\right)\right)$. By  (\ref{ldos}), $u_{\gamma}\rightarrow u$ in $ C\left(\left[0,T\right];L^{2}\left(\mathbb{T}^{2}\right)\right)$. Hence, by weak* compactness $u\;\in L^{\infty}\left(\left[0,T\right];H^{s}\left(\mathbb{T}^{2}\right)\right)$. To establish the
uniqueness of $u$ we follow a similar argument as above and the continuous dependence one uses the Bona-Smith argument(see\cite{bona1975initial}).

\begin{remark}
	
	Following the ideas of \cite{duqueversion}, we obtain a better result of local and global well-posedness 
	to the regularized problem  (\ref{regularizado2});   
	\begin{equation} \label{regularizado2}
		\begin{cases}
			u_{t}-\partial_{x}\left(D_{x}^{1+\alpha}u\pm D_{y}^{1+\beta}u\right)+uu_{x}+\mu\Delta^{2}u=0 & \left(x,y\right)\in\mathbb{T}^{2},\: t\in\mathbb{R} \\
			u\left(x,y,0\right)=\phi\left(x,y\right)&  \mu>0.
		\end{cases}
	\end{equation}
	\begin{enumerate}
		\item For the local well-posedness with  $\mu>0$ and $-2< s<2$.  
		We consider, 
		\begin{align*}
			\chi_{T}^{s}=\left\{ u_{\mu}\in C\left(\left[0,T\right];H^{s}\left(\mathbb{T}^{2}\right)\right):\left\Vert u\right\Vert _{\chi_{T}^{s}}<\infty\right\}
		\end{align*}
		with $\left\Vert u\right\Vert _{\chi_{T}^{s}}=\underset{\left[0,T\right]}{sup}\left\{ \left\Vert u\left(t\right)\right\Vert _{s}+t^{\frac{s}{4}}\left\Vert u\right\Vert _{0}\right\} $ .
		We apply the fixed point theorem and  respective linear and nonlinear estimates; 
		\begin{align*}
			\left\Vert \mathbb{W}_{\mu}\left(t\right)\phi\right\Vert _{\chi_{T}^{s}}&\leq\left\Vert \mathbb{W}_{\mu}\left(t\right)\phi\right\Vert _{s}+t^{\frac{s}{4}}\left\Vert \mathbb{W}_{\mu}\left(t\right)\phi\right\Vert _{0}\\&\leq\left( 1+C_{s}\left(T^{\frac{\left|s\right|}{4}}+\mu^{-\frac{\left|s\right|}{4}}\right)\right) \left\Vert \phi\right\Vert _{s}
		\end{align*}
		and
		\begin{align*}
			&\left\Vert \int_{0}^{t}\mathbb{W}_{\mu}\left(t-t'\right)\partial_{x}u^{2}dt'\right\Vert _{\chi_{T}^{s}}\\&
			\leq\left\Vert \int_{0}^{t}\mathbb{W}_{\mu}\left(t-t'\right)\partial_{x}u^{2}dt'\right\Vert _{s}+t^{\frac{s}{4}}\left\Vert \int_{0}^{t}\mathbb{W}_{\mu}\left(t-t'\right)\partial_{x}u^{2}dt'\right\Vert _{0}\\& \leq\left( \left(\mu^{-\frac{\left(s+2\right)}{4}}+\mu^{-\frac{1}{2}}\right)T^{\frac{2-\left|s\right|}{4}}+\mu^{-\frac{1}{4}}T^{\frac{3}{4}}\right) \left\Vert u\right\Vert _{\chi_{T}^{s}}^{2}.
		\end{align*}
		
		\item The regularized problem  (\ref{regularizado2})  satisfies,
		\begin{align}
			\left\Vert \phi\right\Vert _{0}^{2}=\left\Vert u\left(t\right)\right\Vert _{0}^{2}+\mu\int_{0}^{t}\left\Vert \Delta u\right\Vert _{0}^{2}dt',
		\end{align}
		\begin{align}\label{h11dos}
			\left\Vert \partial_{x}u\right\Vert _{0}^{2}\leq\left\Vert \partial_{x}\phi\right\Vert _{0}^{2}e^{\epsilon^{2}T}
		\end{align}
		and
		\begin{align}\label{h11tres}
			\left\Vert \partial_{y}u\right\Vert _{0}^{2}\leq\left\Vert \partial_{y}\phi\right\Vert _{0}^{2}e^{\left(\epsilon^{-2}+\epsilon^{-\frac{4}{3}}\right)T}
		\end{align}
		We get (\ref{h11dos}) and (\ref{h11tres}), applying Gagliardo -Nirenberg inequality (Theorem 3.70 in \cite{aubin} ) 
		for compact surfaces with $\hat{u}(0,0)=0$. with these ingredients we establish  global well-posedness for $s>-2$; 
	\end{enumerate}
	\begin{theorem}
		Let $s\in(-2,2)$ and $\mu>0$. Then (\ref{regularizado2}) is global well-posedness. 
	\end{theorem}
	
\end{remark}
\textbf{Acknowledgements.} Carolina Albarracin acknowledge support from Colciencias-Colombia  by convocatory N°785 of national doctorates. 

\bibliographystyle{amsplain}
\bibliography{RCMBibTeX}

\end{document}